\documentclass[10pt]{amsart}
\usepackage[margin=1in]{geometry}
\usepackage{amsmath,amssymb,setspace}
\usepackage{indentfirst}
\usepackage{comment}
\newcommand{\ints}{\mathbb{Z}}

\newcommand{\cplx}{\mathbb{C}}

\renewcommand{\phi}{\varphi}

\DeclareFontFamily{OT1}{rsfs}{}
\DeclareFontShape{OT1}{rsfs}{n}{it}{<-> rsfs10}{}
\DeclareMathAlphabet{\mathscr}{OT1}{rsfs}{n}{it}

\newtheorem{theorem}{Theorem}
\newtheorem{corollary}[theorem]{Corollary}
\newtheorem{lemma}[theorem]{Lemma}

\newtheorem{definition}[theorem]{Definition}

\newcommand{\e}{\tilde{e}}
\newcommand{\f}{\tilde{f}}
\newcommand{\Ch}{\operatorname{Ch}}

\newcommand{\B}{\mathcal{B}(\infty)}
\renewcommand{\L}{\mathcal{L}(\infty)}
\renewcommand{\binom}[2]{\genfrac{[}{]}{0pt}{}{#1}{#2}}
\DeclareMathOperator{\wt}{wt}

\begin{document}

\title [The $\B$ Crystal for a Certain Family of Generalized Quantum Groups]{An Explicit Description of the $\mathcal{B}(\infty)$ Crystal for Generalized Quantum Groups of a Family of Comet Quivers} \author{Uma Roy and Seth Shelley-Abrahamson}\date{\today}\maketitle

\begin{abstract}

	In \cite{Bo1} Tristan Bozec gave a definition of \emph{generalized quantum groups} that extends the usual definition of quantum groups to finite quivers with loops at vertices, and in \cite{Bo2} he introduced a theory of \emph{generalized crystals} for this new family of Hopf algebras. We explicitly characterize the generalized crystal $\mathcal{B}(\infty)$ associated to a certain family of quivers with multiple loops by providing a complete set of relations among the Kashiwara operators themselves.
	
\end{abstract}

\tableofcontents

\section {Introduction} 

Classically, as in \cite{Lu}, to a quiver without loops at a vertex one has an associated quantum group.  More recently, by considering certain categories of perverse sheaves on quiver varieties, Tristan Bozec has introduced a generalization of quantum groups associated instead to an arbitrary finite quiver, possibly with loops.  See \cite{Bo1} for full details.  In particular, Bozec associates to any finite quiver $Q$ a Hopf algebra $U = U(Q)$ over $\cplx(v)$ which coincides with the usual definition of quantum group in the case $Q$ has no loops.  $U$ shares many properties with usual quantum groups, and in particular admits a triangular decomposition $U = U^- \otimes U^0 \otimes U^+$ as in the classical case.  In analogy with the theory of crystal bases developed in \cite{Ka} for usual quantum groups, in \cite{Bo2} Bozec introduces a theory of generalized crystals for an arbitrary quiver, and in particular gives a construction of a crystal basis $(\mathcal{L}(\infty), \mathcal{B}(\infty))$ for the negative part $U^-$ of the generalized quantum group.  As in the case of quivers without loops, the crystal $\mathcal{B}(\infty)$ carries much information about the algebra $U$ itself and about its representation theory.

In this paper we study the generalized crystal $\mathcal{B}(\infty)$ associated to certain quivers called \emph{comet-shaped quivers} (with leg lengths equal to 1), which appear, for example, in \cite{HaLeRo}.  The Hopf algebra $U$ and the crystal basis $(\mathcal{L}(\infty), \mathcal{B}(\infty))$ do not depend on the orientation of the underlying quiver, so we will not specify orientations of quivers in this paper.  Given integers $\omega \geq 1$ and $r \geq 0$, let $Q(\omega, r)$ denote the quiver with vertices $\{i, j_1, ..., j_r\}$ with $\omega$ loops at vertex $i$, no loops at or edges connecting vertices $j_1, ..., j_r$, and exactly one edge connecting the vertices $i$ and $j_s$ for $1 \leq s \leq r$.  We will provide in the ``non-isotropic'' case $\omega > 1$ a remarkably simple description of the generalized crystal $\mathcal{B}(\infty)$ associated to the quiver $Q(\omega, r)$ for any $r \geq 1$ by giving a complete set of relations among the corresponding Kashiwara operators on $\mathcal{B}(\infty)$ defined in the following section.  These relations should be seen as degenerations of the commutation and Serre relations defining $U^-$.  

\section{Acknowledgments} 

	The authors would like to thank Tristan Bozec for suggesting this project and for many useful discussions, and Pavel Etingof for helpful discussion that led to the proof of Lemma \ref{triple-binom}. Author Uma Roy would like to thank the MIT-PRIMES program and Research Science Institute for facilitating a part of this research. 

\section{Background and Definitions} In this section we recall the relevant definitions from \cite{Bo1} and \cite{Bo2}.

\subsection{The Algebra $U^-$}\label{U-} Fix a quiver $Q$, possibly with loops, with vertex set $I$.  Let $$(\cdot, \cdot) : \ints I \times \ints I \rightarrow \ints$$ be the unique symmetric bilinear form on the free abelian group $\ints I$ with $(i, i) = 2 - 2\omega_i$, where $\omega_i$ is the number of loops at vertex $i \in I$, and with $(i, j) = -n_{ij}$, where $n_{ij}$ is the number of edges connecting the vertices $i, j \in I$.  A vertex $i \in I$ is called \emph{real} if there are no loops at $i$, and otherwise is called \emph{imaginary}.  We denote by $I^{re}$ the set of real vertices, and by $I^{im}$ the set of imaginary vertices.  An imaginary vertex $i \in I^{im}$ is called \emph{isotropic} if $\omega_i = 1$, and \emph{non-isotropic} otherwise.  We denote by $I^{iso} \subset I^{im}$ the set of imaginary isotropic vertices.  Define $$I_\infty := I^{re} \cup \{(i, l) : i \in I^{im}, l \geq 1\}$$ and extend the pairing $(\cdot, \cdot)$ by defining $(j, (i, l)) = ((i,l),j)= l(j, i)$ for $j \in I^{re}$, $i \in I^{im}$, and $l \geq 1$ and $((j, k), (i, l)) = kl(j, i)$ for $i, j \in I^{im}$ and $k, l \geq 1$.

Let $A = \cplx(v)\langle F_{\iota} : \iota \in I_\infty\rangle$ be the free $\cplx(v)$-algebra on the generators $F_{\iota}$ for $\iota \in I_\infty$.  We give $A$ a $\ints I$ grading by setting $\deg(F_j) = -j$ and $\deg(F_{(i, l)}) = -li$ for $j \in I^{re}$, $i \in I^{im}$, and $l \geq 1$.  For an integer $n \in \ints$, define the $v$-analogue of $n$ by $$[n] := \frac{v^n - v^{-n}}{v - v^{-1}} = v^{n - 1} + v^{n - 3} + \ldots + v^{1 - n}$$ and, for $n \geq 1$, the $v$-analogue of $n!$ by $$[n]! := [n] \cdots [1].$$  We set $[0]! = 1$.  For $k, n \in \ints$ with $k \geq 0$, the $v$-analogue of the binomial coefficient $n \choose k$ is defined by $$\binom{n}{k} := \frac{[n][n-1]\cdots[n-k+1]}{[k]!}.$$ For a real vertex $j \in I^{re}$ and an integer $n \geq 0$, define its $n^{th}$ \emph{divided power} by $$F_j^{(n)} = \frac{1}{[n]!}F_j^n.$$  We define $U^-$ as the quotient of the free algebra $A$ by the ideal generated by the relations $$[F_{\iota}, F_{\kappa}] = 0$$ for all $\iota, \kappa \in I_\infty$ with $(\iota, \kappa) = 0$ and $$\sum_{t + t' = 1 - (j, \iota)} (-1)^tF_j^{(t)}F_{\iota}F_j^{(t')} = 0$$ for all $j \in I^{re}$ and $\iota \in I_\infty$ with $\iota \neq j$.  The relations of the first type are called \emph{commutation relations} and the relations of the second type are called \emph{Serre relations}.  Note that both types of relations are homogeneous, so $U^-$ inherits a $\ints I$-grading from $A$.

For a homogeneous element $u \in U^-$, let $|u| = \deg(u) \in \ints I$ denote its degree.  For $d \in \ints I$ let $$U^-[d] := \{u \in U^- : |x| = d\}$$ denote the homogeneous subspace of $U^-$ of degree $d$.

\subsection{Kashiwara Operators and the Crystal $\mathcal{B}(\infty)$} \label{2.2} In \cite{Bo2}, Proposition 3.11, Bozec defines certain elements $b_{i, l} \in U^-[-li]$ for all imaginary vertices $i \in I^{im}$ and positive integers $l \geq 1$.  We do not fully reproduce their definition here, but we recall some of their relevant properties.  In particular, we have $b_{i, 1} = F_{i, 1}$ and $$b_{i, l} - F_{i, l} \in \cplx(v)\langle F_{i, k} : 1 \leq k < l\rangle.$$

For an imaginary vertex $i \in I^{im}$ and a nonnegative integer $l \geq 1$, if $i$ is isotropic let $\mathcal{C}_{i, l}$ denote the set of partitions of $l$, and otherwise let $\mathcal{C}_{i, l}$ denote the set of compositions of $l$.  Let $\mathcal{C}_i := \coprod_{l \geq 0} \mathcal{C}_{i, l}$.  We denote partitions or compositions by finite lists of the form $c = (c_1, c_2, ...)$, where for partitions these lists are unordered.  For such $c = (c_1, c_2, ..., c_k) \in \mathcal{C}_i$, let $b_{i, c} = b_{i, c_1} \cdots b_{i, c_k}.$  Observe that $\{b_{i, c} : c \in \mathcal{C}_{i, l}\}$ forms a basis for $U^-[-li]$.  For convenience, if $j \in I^{re}$ is a real vertex, let $b_j = F_j$, so that we have defined $b_{\iota}$ for all $\iota \in I_\infty$.  Then we observe that the set $\{b_{\iota} : \iota \in I_\infty\}$ generates $U^-$ as an algebra and we will see in Corollary \ref{algebrahom} that the assignment $F_{\iota} \mapsto b_{\iota}$ extends to an algebra endomorphism of $U^-$. 

By Proposition 3.14 in \cite{Bo2}, for each $\iota \in I_\infty$ there exists a unique $\cplx(v)$-linear function $e_{\iota}' : U^- \rightarrow U^-$ characterized by the properties: $$(1) \ \ \ \ e_{\iota}'(yz) = e_\iota'(y)z + v^{(-\iota, |y|)}ye_\iota'(z) \ \ \ \ \forall y, z \in U^-$$ $$(2) \ \ \ \ e_\iota'(b_\kappa) = \delta_{\iota, \kappa} \ \ \ \ \forall \kappa \in I_\infty.$$  For a real vertex $j \in I^{re}$, let $\mathcal{K}_j = \ker(e_j')$, and for an imaginary vertex $i \in I^{im}$ let $\mathcal{K}_i = \cap_{l \geq 1} \ker(e_{(i, l)}').$  We then have the following proposition, which is Proposition 16 in \cite{Bo2}:

\begin{lemma}\label{decomp} For a real vertex $j \in I^{re}$ there is a direct sum decomposition $$U^- = \bigoplus_{l \geq 0} F_j^{(l)}\mathcal{K}_j$$ and for an imaginary vertex $i \in I_\infty$ there is a direct sum decomposition $$U^- = \bigoplus_{c \in \mathcal{C}_i} b_{i,c}\mathcal{K}_i.$$\end{lemma}

We can now define the Kashiwara operators $\tilde{e}_\iota, \tilde{f}_\iota : U^- \rightarrow U^-$ for $\iota \in I_\infty$ as in Definition 3.17 of \cite{Bo2}.  First suppose $j \in I^{re}$ is a real vertex.  For $u \in U^-$, by Lemma \ref{decomp}, we can write uniquely $u = \sum_{l \geq 0} F_j^{(l)}z_l$ with $z_l \in \ker{e_j'}$.  The Kashiwara operators $\tilde{e}_j, \tilde{f}_j$ are then defined by $$\tilde{f}_j(u) = \sum_{l \geq 0} F_j^{(l + 1)}z_l \ \ \ \ \ \tilde{e}_j(u) = \sum_{l \geq 1} F_j^{(l - 1)}z_l.$$  Next, suppose $i \in I^{im}$ is an imaginary vertex.  Given $u \in U^-$, write $u = \sum_{c \in \mathcal{C}_i} b_{i,c} z_c$ with $z_c \in \mathcal{K}_i$ for all $c \in \mathcal{C}_i$ as in Lemma \ref{decomp}.  Then we define $$\tilde{f}_{i, l}(u) := \begin{cases} 
      \displaystyle\sum_{c \in \mathcal{C}_i} b_{i, (l, c)}z_c & i \notin I^{iso} \\
      \displaystyle\sum_{\lambda \in \mathcal{C}_i} \sqrt{\frac{l}{m_l(\lambda) + 1}}b_{i, \lambda \cup l}z_\lambda & i \in I^{iso} 
   \end{cases}$$ and $$\tilde{e}_{i, l}(u) := \begin{cases} 
      \displaystyle\sum_{c \in \mathcal{C}_i: c_1 = l} b_{i, c\backslash c_1}z_c & i \notin I^{iso} \\
      \displaystyle\sum_{\lambda \in \mathcal{C}_i : l \in \lambda} \sqrt{\frac{m_l(\lambda)}{l}}b_{i, \lambda \backslash l}z_\lambda & i \in I^{iso} 
   \end{cases}$$ where if $c = (c_1, c_2, \ldots)$ is a composition then $c\backslash c_1 = (c_2, c_3, \ldots)$ and $(l, c) = (l, c_1, c_2, \ldots)$, and if $\lambda = (\lambda_1, \lambda_2, \ldots)$ is a partition then $m_l(\lambda)$ is the number of parts of $\lambda$ equal to $l$, $\lambda \backslash l$ denotes the partition obtained from $\lambda$ by removing a part of size $l$, and $\lambda \cup l$ denotes the partition obtained from $\lambda$ by adding a part of size $l$.
   
   Let $\mathcal{A} \subset \cplx(v)$ denote the subring consisting of rational functions in $v$ without pole at $v^{-1} = 0$, in other words the localization of $\cplx[v^{-1}]$ at the maximal ideal $(v^{-1})$.  Let $\mathcal{L}(\infty)$ denote the sub-$\mathcal{A}$-module of $U^-$ spanned by the elements $\tilde{f}_{\iota_1} \cdots \tilde{f}_{\iota_s}.1$ where $s \geq 0$ and $\iota_k \in I_\infty$ for $1 \leq k \leq s$.  Finally, we define the set $$\mathcal{B}(\infty) := \{\tilde{f}_{\iota_1} \cdots \tilde{f}_{\iota_s}.1 : \iota_k \in I_\infty\} \subset \frac{\mathcal{L}(\infty)}{v^{-1}\mathcal{L}(\infty)}.$$  We have the following theorem, which is Theorem 3.26 of \cite{Bo2}:
   
 \begin{theorem} The Kashiwara operators $\tilde{e}_\iota, \tilde{f}_\iota$ for $\iota \in I_\infty$ are still defined on $\mathcal{B}(\infty)$, and there are functions $\wt: \mathcal{B}(\infty) \rightarrow \ints I$ and $\epsilon_i : \mathcal{B}(\infty) \rightarrow \mathcal{C}_i \cup \{-\infty\}$ such that $B(\infty)$ together with these maps forms a generalized $Q$-crystal in the sense of Definition 3.18 of \cite{Bo2}.\end{theorem}
 
Note that the Kashiwara operators $\tilde{f}_{\iota}, \tilde{e}_\iota : U^- \rightarrow U^-$ for $\iota \in I_\infty$ are graded operators of degrees $-j$ and $+j$, respectively, for $\iota = j \in I^{re}$ and are graded operators of degrees $-li$ and $+li$, respectively, for $\iota = (i, l)$ with $i \in I^{im}$ and $l \geq 1$.  In particular, $\mathcal{L}(\infty)$ and hence $\mathcal{B}(\infty)$ inherit $\ints I$-gradings as well, and the Kashiwara operators are graded operators of the same degrees as on $U^-$.  In the special cases of the comet quivers $Q(\omega, r)$ defined in the introduction, we describe the $\ints I$-graded set $B(\infty)$ and the Kashiwara operators defined on it explicitly in terms of sequences with special properties.

\section{Relations in $\B$ Corresponding to Non-isotropic Comet Quivers}\label{algebra}

	Recall from the introduction that $Q(\omega, r)$ is the quiver with vertex set $I = \{ i,j_1, \ldots, j_r\}$ with $\omega$ loops at the imaginary vertex $i$ and exactly 1 edge pairwise connecting $i$ and real vertices $j_1, \ldots, j_r$. Writing $j$ without a subscript refers to any real vertex. In this paper we deal with only the non-isotropic case ($\omega > 1$). To state our main theorem, we first define \emph{crystal Serre relations} among elements $\f_\iota$ and $\f_{j_k}$.

\begin{definition}[Crystal Serre Relation]	
	For $\iota \in I_\infty$ and a real vertex $j \in I^{re}$, say that $\f_\iota$ satisfies the $l$-th order \emph{crystal Serre relation} with $\f_j$ if $\f_j \f_\iota  \f_j^l \equiv \f_\iota \f_j^{l+1}$ as operators on $\B$.

		
\end{definition}

\begin{definition}[Commutation Relation]
	
	For $\iota, \iota' \in I_\infty$, say that that $\f_\iota$ and $\f_{\iota'}$ satisfy the \emph{commutation relation} if $\f_\iota \f_{\iota'} \equiv \f_{\iota'} \f_\iota$ as operators on $\B$. 

\end{definition}

\begin{theorem}[Main Theorem] \label{maintheorem}
	
	In the generalized crystal $\B$ associated to the quiver $Q(\omega, r)$ for $\omega > 1$ and $r \geq 0$, the operator $\f_{(i, l)}$ satisfies the $l$-th order crystal Serre relation with the operators $\f_{j_k}$ for all $1 \leq k \leq r$, and the operators $\f_j$ and $\f_{j'}$ commute for any $j, j' \in \{j_1, ..., j_r\}$.  Furthermore, every equality $\f_{\iota_1} \ldots \f_{\iota_n}.1 \equiv \f_{\iota'_1} \ldots \f_{\iota'_n}.1$ in $\B$ follows from the commutation relations and crystal Serre relations. 

\end{theorem}

In this section we show that the relations in the theorem hold, and in Section \ref{combo_charform} we will show by a combinatorial argument with Bozec's character formula for $U^-$ (given in \cite{addendum}) that the these relations imply all equalities in $\mathcal{B}(\infty)$ among compositions of Kashiwara operators applied to $1$, giving a complete description of $\B$. 


Recall from Section \ref{U-} the Serre relation, which states for all $\iota \in I_\infty$, $\iota \neq j$,

\begin{equation}\label{Serre-Relation}\sum_{t + t' = 1 - (j, \iota)} (-1)^tF_j^{(t)}F_{\iota}F_j^{(t')} = 0.\end{equation}

We define the notion of an \emph{a}-th order Serre relation for any element $x \in U^-$. 

\begin{definition}

	An element $x \in U^-$ satisfies the \emph{$a$-th order Serre relation with $F_j$} if 
	
	$$\sum_{t=0}^{a+1} (-1)^{a+1-t}F_j^{(a+1-t)}xF_j^{(t)} = 0.$$
	
\end{definition}

\begin{lemma}\label{moving-Fjs}
	Let $x \in U^-$ satisfy the $l$-th order Serre relation with $F_j$. Then for $n \in \mathbb{N}$, 

	$$xF_j^{(l+1+n)} = \sum_{t=0}^{l} (-1)^{l+t} \binom{l+n-t}{n} F_j^{(l+1+n-t)} x F_j^{(t)}.$$
\end{lemma}

\begin{proof}
	
	We induct on $n$. When $n=0$, $(-1)^{l+t} \binom{l+n-t}{n} = (-1)^{l+t}$, so the above is exactly the $l$-th order Serre relation of $x$ with $F_j$.
	
	Now take $n>0$. By induction on $n$ we have
	$$xF_j^{(l+1+n)}  = \frac{1}{[l+1+n]} xF_j^{(l+n)} F_j $$
	$$= \frac{1}{[l+1+n]} \sum_{t=0}^{l} (-1)^{t+l} \binom{l+n-1-t}{n-1} F_j^{(l+n-t)} x F_j^{(t)} F_j$$
	
	$$ = \frac{1}{[l+1+n]} \left( \sum_{t=0}^{l-1} [t+1] (-1)^{t+l} \binom{l+n-1-t}{n-1} F_j^{(l+n-t)}x F_j^{(t+1)} + [l+1] \binom{n-1}{n-1}F_j^{(n)}  xF_j^{(l+1)} \right)$$
	
	$$ =  \frac{1}{[l+1+n]} \left( \sum_{t=0}^{l-1} [t+1] (-1)^{t+l} \binom{l+n-1-t}{n-1} F_j^{(l+n-t)}x F_j^{(t+1)} + [l+1] F_j^{(n)} \sum_{t=0}^{l} (-1)^{l-t} F_j^{(l+1-t)}x F_j^{(t)} \right).$$
	
	To prove the lemma, it suffices to show that
	
	$$\frac{[t]}{[l+1+n]} (-1)^{t-1+l} \binom{l+n-t}{n-1}+ (-1)^{l-t} \frac{[l+1]}{[l+1+n]} \binom{l+1+n-t}{n} = (-1)^{t+l} \binom{l+n-t}{n}.$$
	
	Simple algebraic manipulation reveals this equality is equivalent to proving
	
	
	
	$$ [l+1+n][l+1-t] + [t][n] = [l+1][l+1+n-t],$$
	
	which can readily be checked using the definition of quantum numbers.
\end{proof}


\begin{lemma}\label{triple-binom}
	For $a,b,c \geq 0$,
	
	\begin{equation}\label{triple-identity}\sum_{s=0}^{c} (-1)^{s} \binom{b+s}{s}  \binom{a+s}{c} \binom{b+c+1}{c-s} = \binom{a-b-1}{c}.\end{equation}
	
\end{lemma}

\begin{proof}
	
	To prove this identity, we transform Equation \eqref{triple-identity} to the $q$-analogue of the natural numbers defined for $n \geq 0$ by
	$$[n]_q = \frac{1-q^n}{1-q}.$$
	
	The $q$-analogue of $n!$ is defined as $[n]_q! := [n]_q\cdots [1]_q$ and the $q$-analogue of the binomial coefficient is defined by
	
	$$\binom{n}{k}_q = \frac{[n]_q \ldots [n-k+1]_q}{[k]_q!}.$$
	
	For $q = v^2$ we have $[n]_v = q^{\frac{1-n}{2}}[n]_q$, so Equation \eqref{triple-identity} becomes equivalent to showing
	
	\begin{equation}\label{qtripsum}\sum_{s=0}^{c} (-1)^{s} q^{\frac{s^s+s}{2} - cs}  \binom{b+s}{s}_q  \binom{a+s}{c}_q \binom{b+c+1}{c-s}_q = q^{bc+c}\binom{a-b-1}{c}_q.\end{equation}

	To prove Equation \eqref{qtripsum}, we view both sides as polynomials in the ring $Q(a,b) := \mathbb{Q}(q)[q^a, q^b]$,
	 where $q^a$ and $q^b$ are formal variables. Viewing $\binom{a+s}{c}_q$ as a polynomial in the variable $q^a$, we claim that for $0 \leq s \leq c$ the polynomials $\binom{a+s}{c}$ form a basis for $\mathbb{Q}(q)[q^a]_{\leq c}$, the subspace of $\mathbb{Q}(q)$ of polynomials of degree at most $c$. Assume there is some non-trivial relation between $\binom{a+s}{c}_q$, such that $\sum_{s=0}^c y_s \binom{a+s}{c}_q = 0$. Evaluating the $q$-binomial coefficient $\binom{a+s}{c}_q$ at $a < c-s$ gives 0 and at $a=c-s$ gives 1. It follows that the $\binom{a+s}{c}_q$ form a basis of the free $\mathbb{Q}(q)[q^b]$-module $\mathbb{Q}(q)[q^b] [q^a]_{\leq c}$ (the $\leq c$ referring to the $a$ degree). In particular, there exist $x_s \in \mathbb{Q}(q)[q^b]_{\leq c}$ such that 
	 
	
	
	\begin{equation}\label{xs}\sum_{s=0}^c x_s \binom{a+s}{c} = q^{bc+c}\binom{a-b-1}{c}.\end{equation}
	
	For a particular fixed value of $s$, we see for $b \in \{-c-1, \ldots, -1 \}$, $x_s$ evaluated at $b \neq -s-1$ gives 0, and $x_s$ evaluated at $b=-s-1$ gives $q^{bc+c}$. Thus $(q^b - q^{-c-1})(q^b - q^{-c})\cdots(q^b - q^{-1})/(q^b - q^{-s-1})$ divides $x_s$. Multiplying through by the appropriate constants in $\mathbb{Q}(q)$, we see that $$x_s = K_s \binom{b+s}{s} \binom{b+c+1}{c-s},$$ where $K_s$ is some constant in $\mathbb{Q}(q)$. Evaluating $x_s$ at $b = -s-1$ gives that $x_s(q^{-s-1}) = q^{(-s-1)c+c} = \binom{-1}{s}_q K_s$. Solving gives $K_s = (-1)^s q^{\frac{s^2+s}{2} - cs}$. 
	
	Thus $$x_s = q^{\frac{s^s+s}{2} - cs}  \binom{b+s}{s}_q  \binom{b+c+1}{c-s}_q,$$ and the identity in Equation \eqref{qtripsum} and the lemma are proven.
\end{proof}

Having proven Lemma \ref{triple-binom}, we can now show that all homogeneous elements of $U^-$ satisfy Serre relations with all real vertices, with the order of the Serre relation depending only on the homogeneous degree.

\begin{theorem}\label{gen-serre}
	Given elements $x,y \in U^-$ satisfying respectively the $a$-th order and $b$-th order Serre relations with $F_j$, then $xy$ satisfies the $(a+b)$-th order Serre relation with $F_j$.
\end{theorem}

\begin{proof}	
	$$\sum_{t=0}^{a+b+1} (-1)^{(1+a+b-t)}  F_j^{(1+a+b-t)}xy F_j^{(t)}$$ 
	
	
	$$=\sum_{t=0}^b (-1)^{1+a+b-t} F_j^{(1+a+b-t)} xy F_j^{(t)} + \sum_{s=0}^{a} (-1)^{(1+a+b-(s+b+1))} F_j^{(1+a+b-(s+b+1))} xy F_j^{(s+b+1)}.$$
	
	By Lemma \ref{moving-Fjs}, we have
	
	$$=\sum_{t=0}^b (-1)^{1+a+b-t} F_j^{(1+a+b-t)} xy F_j^{(t)} + \sum_{s=0}^{a} (-1)^{(a-s)} F_j^{(a-s)} x \sum_{k=0}^{b} (-1)^{(b+k)} \binom{b+s-k}{s} F_j^{(b+1+s-k)} y F_j^{(k)}$$
	
	$$ = \sum_{t=0}^b \left[  (-1)^{1+a+b-t} F_j^{(1+a+b-t)} x   + \sum_{s=0}^{a} (-1)^{(a+b+t-s)} \binom{b+s-t}{s} F_j^{(a-s)} x  F_j^{(b+1+s-t)} \right ] yF_j^{(t)}$$ 
	
	It suffices to show that the expression in the brackets in the line above is $0$ for $0 \leq t \leq b$.
	
	
	To use the $a$-th order Serre relation for $x$, we must have that $b+1+s-t \geq a+1$ i.e. $s \geq a-b+t$. 
	
	$$(-1)^{1+a+b-t} F_j^{(1+a+b-t)} x   + \sum_{s=0}^{a} (-1)^{(a+b+t-s)} \binom{b+s-t}{s} F_j^{(a-s)} x  F_j^{(b+1+s-t)}$$
	

	$$ =  (-1)^{1+a+b-t} F_j^{(1+a+b-t)} x   + \sum_{s=0}^{a-b+t-1} (-1)^{(a+b+t-s)} \binom{b+s-t}{s} F_j^{(a-s)} x  F_j^{(b+1+s-t)}$$ 
	$$+ \sum_{s=0}^{b-t} (-1)^{(a+b+t-(s+a-b+t))} \binom{b+(s+a-b+t)-t}{s+a-b+t} F_j^{(a-(s+a-b+t))} x  F_j^{(b+1+(s+a-b+t)-t)} $$

	$$ =  (-1)^{1+a+b-t} F_j^{(1+a+b-t)} x   + \sum_{s=0}^{a-b+t-1} (-1)^{(a+b+t-s)} \binom{b+s-t}{s} F_j^{(a-s)} x  F_j^{(b+1+s-t)}$$ 
	$$+ \sum_{s=0}^{b-t} (-1)^{s} \binom{s+a}{s+a-b+t} F_j^{(b-t-s)} \sum_{i=0}^{a} (-1)^{a+i} \binom{a+s-i}{s} F_j^{(a+s+1-i)}x  F_j^{(i)} $$

	$$ =  (-1)^{1+a+b-t} F_j^{(1+a+b-t)} x   + \sum_{s=0}^{a-b+t-1} (-1)^{(a+b+t-s)} \binom{b+s-t}{s} F_j^{(a-s)} x  F_j^{(b+1+s-t)}$$ 
	$$+ \sum_{i=0}^{a} \sum_{s=0}^{b-t} (-1)^{s+a-i} \binom{a+s-i}{s}  \binom{s+a}{s+a-b+t} \binom{a+b- i - t + 1}{b-t-s} F_j^{(b-t+ a + 1 -i)}x  F_j^{(i)} $$

	To show that this sum is equal to 0, it suffices to show for $0 \leq i \leq a$

	
		$$ (-1)^{(a+i-1)} \binom{i-1}{b-t} + \sum_{s=0}^{b-t} (-1)^{s+a-i} \binom{a+s-i}{s}  \binom{s+a}{s+a-b+t} \binom{a+b- i - t + 1}{b-t-s} = 0$$

		Substituting $c=b-t$ and rearranging, this is equivalent to proving
		
			$$ \sum_{s=0}^{c} (-1)^{s} \binom{a+s-i}{s}  \binom{s+a}{s+a-c} \binom{a+c-i+1}{c-s} =\binom{i-1}{c} $$
		
		Observe that this is exactly the statement in Lemma \ref{triple-binom} with $b = a-i$. The theorem follows.
\end{proof}

\begin{corollary}\label{algebrahom}
	The element $b_{i,l}$ satisfies the order $l$ Serre relation with $b_j = F_{j}$.  In particular, the assignment $F_\iota \mapsto b_\iota$ for $\iota \in I_\infty$ extends to an endomorphism of the algebra $U^-$.
	
\end{corollary}

\subsection{Crystal Serre Relations}

We start by proving the following equality.

\begin{lemma}\label{basicequality}
	
	For all $\ell \geq 1$ and $n \geq 0$, $\f_{i,\ell}\f_j^{\ell+n+1} . 1 = \f_j \f_{i,\ell} \f_j^{\ell+n} . 1$ in $\B$.
	
	
\end{lemma}

To prove this lemma, we prove special properties regarding the decomposition of $b_{i,l}F_j^{(c)}$ into $\bigoplus_{l \geq 0} F_j^{(l)}\mathcal{K}_j$ as given in Lemma \ref{decomp}.

\begin{definition}\label{zkc}
	
	Define $z_{k,c}^{\ell}$ for all $c \geq 0$ and $c \geq k \geq 0$ such that $$b_{i,l}F_j^{(c)} = \sum_{k=0}^{c} F_{j}^{(k)} z_{k,c}^\ell.$$
	
\end{definition}

	This provides a unique definition of $z_{k,c}^{\ell}$ by Proposition 3.16 in Bozec. For $k < 0$ or $c < 0$, we define $z_{k,c}^\ell =0 $. Note that the superscript for $z_{k,c}$ is always $\ell$ for the following sections, thus we often omit the superscript.

\begin{lemma}\label{recursion-lemma}
	For all $k, c \geq 0$ and $\ell \geq 1$, the following recursion holds: 
	
	
	$$z_{k,c} = \frac{1}{[c]} (z_{k,c-1} F_j - v^{-\ell + 2(c-k-1)} F_j z_{k,c-1}^\ell + [k]v^{-\ell + 2(c-k)} z_{k-1,c-1}).$$

\end{lemma}

\begin{proof}

	We observe that $|z_{k,c}| = -\ell i - (c-k)j$, as $|z_{k,c}^\ell| + |F_j^{(k)}| = |b_{i,l}F_j^{(c)}|$ . Thus $-(|z_{k,c}|, j)  = -\ell + 2 (c-k)$.

	From the proof of Proposition 3.16 in \cite{Bo2}, we can write
	$$z_{k,c-1} F_j = (z_{k,c-1} F_j - v^{-(|z_{k,c-1}|,j)}F_jz_{k,c-1}) + v^{-(|z_{k,c-1}|, j)}F_jz_{k,c-1},$$
	where the first term in parenthesis lies in $\mathcal{K}_j$.

	$$\sum_{k=0}^{c} F_{j}^{(k)} z_{k,c} = b_{i,l}F_j^{(c)} = \frac{1}{[c]} b_{i,l}F_j^{(c-1)} F_j = \frac{1}{[c]} \sum_{k=0}^{c-1} F_j^{(k)} z_{k,c-1} F_j$$
	$$= \frac{1}{[c]} \sum_{k=0}^{c-1} F_j^{(k)} ((z_{k,c-1} F_j - v^{-\ell + 2(c-k-1)}F_jz_{k,c-1}) + v^{-\ell + 2(c-k-1)}F_jz_{k,c-1}).$$
		
	The lemma now follows from the uniqueness of the decomposition in Propsition 3.16 of \cite{Bo2}.
\end{proof}

\begin{lemma}
	For all $k,c \geq  0$ and $\ell \geq 1$, $z_{k,c}^\ell = v^{k(c-k-l)} z_{0,c-k}^\ell$. 
\end{lemma}

\begin{proof}

	We define $z'_{0,0} = b_{i,l}$ and define $z'_{0,c} = \frac{1}{[c]} (z'_{0,c-1} F_j - v^{(-\ell +  2(c-1))} F_jz'_{0,c-1})$ for $c > 0$. Then we define $z'_{k,c}$ for $0 < k \leq c$ by $z'_{k,c} = v^{k(c-k-l)}z'_{0, c-k}$. It follows from Lemma \ref{recursion} and the fact that $z_{0,0} = b_{i,l}$ that $z_{0,c}' = z_{0,c}$. We need only to check that the $z_{k,c}'$ satisfy the recurrence of Lemma \ref{recursion}. This is checked in the following calculation. 
	
	$$\frac{1}{[c]} (z'_{k,c-1}F_j - v^{(-\ell + 2(c-k-1))}F_jz'_{k,c-1} + [k]v^{-\ell + 2(c-k)}z'_{k-1,c-1})$$
	
	$$=\frac{1}{[c]} (v^{k(c-1-k-\ell)}z'_{0,c-1-k} F_j - v^{(-\ell + 2(c-k)-1)} F_j v^{k(c-1-k-\ell)}z'_{0,c-1-k} + [k] v^{-\ell + 2(c-k)} v^{(k-1)(c-k-\ell)}z'_{0, c-k})$$ 
	
	$$= \frac{1}{[c]} (v^{k(c-1-k-\ell)} (z'_{0,c-1-k}F_j - v^{(-\ell + 2(c-k)-1)} F_j z'_{0,c-1-k}) + [k]v^{-\ell + 2(c-k)} v^{-(c-k) + l  + k(c-k-l)}z'_{0, c-k}) $$

	$$ = \frac{1}{[c]} ( [c-k] v^{- k + k(c-k-\ell)} z'_{0,c-k} + [k]v^{(c-k + k(c-k-\ell))} z'_{0,c-k})$$

	$$ = \frac{1}{[c]} ([c-k]v^{-k} z'_{k,c} + [k]v^{c-k}z'_{k,c})$$
	
	$$= \frac{[c-k]v^{-k} + [k]v^{c-k}}{[c]} z'_{k,c}$$
	
	$$= z'_{k,c}.$$
\end{proof}

\begin{lemma} \label{inLinfty}
	For $c \leq \ell$, $z_{0,c}^\ell \equiv \f_{i,l} \f_j^c .1$. 
\end{lemma}

\begin{proof}

	We use induction on $c$. For $c=0$, $z_{0,c}^\ell = b_{i,l} = \f_{i,l} . 1 \in \mathcal{L} (\infty)$. 
	
	For $c>0$, by definition we have $$b_{i,l}F_j^{(c)} = \sum_{k=0}^{c} F_{j}^{(k)} z_{k,c} $$
	
	$$\Rightarrow z_{0,c} = \f_{i,l} \f_j^c . 1 - \sum_{k=1}^{c} F_{j}^{(k)} z_{k,c}.$$
	
	$$= \f_{i,l} \f_j^c . 1 - \sum_{k=1}^{c} F_{j}^{(k)} v^{k(c-k-\ell)}z_{0,c-k}$$
	
	Since $z_{0, c-k} \in \mathcal{K}_j$,  $F_{j}^{(k)} z_{0,c-k} = \f_j^k . z_{0, c-k}$.
	
	Thus we have 
	
	$$z_{0,c} = \f_{i,l} \f_j^c . 1 - \sum_{k=1}^{c} v^{k(c-k-\ell)} \f_j^k .  z_{0,c-k}$$

	$v^{k(c-k-l)} \f_j^k z_{0,c-k} \in v^{-1}\L$ because $k(c-k-l) < 0 $ for $0 < k \leq c \leq l$ and $\L$ is stable under $\f_j^k$. Thus $z_{0,c} \equiv \f_{i,l}\f_j^c.1$ in $\B$. 
\end{proof}


Now we prove a particular identity involving quantum binomial coefficients that occurs in the proof of Lemma \ref{basicequality}.

\begin{lemma}\label{u_{n,t}sum}
	For $r \geq 0$ and $n \geq 1$, the following quantum binomial identity holds 
	
	$$\sum_{k=0}^{r} v^{-kr}  (-1)^{k-r} \binom{n-k+r}{n} \binom{r+n+1}{k} = v^{-r(n+r+1)}.$$
\end{lemma} 

\begin{proof}

	We use induction on $r$ to show this claim, where the identity for $r=0$ is obvious. 
	
	First we note for all $r, n \geq 0$, 
	
	\begin{equation}\label{=0} \sum_{k=0}^{r} v^{-kr}  (-1)^{k-r} \binom{n-k+r}{n} v^k \binom{r+n}{k}\end{equation}
	
	\begin{equation}\label{=02} \frac{[n+r]!}{[n]! [r]!}\sum_{k=0}^{r} v^{-k(r-1)}  (-1)^{k-r} \binom{r}{k} = 0, \end{equation}
	
	where the last equality follows from a well known quantum identity. 
	
	By the Pascal identity, which states 
	
	$$\binom{r+n+1}{k} = v^k \binom{r+n}{k} + v^{-r-n+k-1} \binom{r+n}{k-1},$$
	
	we see that the original summation identity is equivalent to considering 
	
	$$\sum_{k=1}^{r} v^{-kr} (-1)^{k-r} \binom{n-k+r}{n} v^{-r-n+k-1}\binom{r+n}{k-1}$$

	$$ =v^{-r-n-1}\sum_{k=1}^{r} v^{-k(r-1)} (-1)^{k-r} \binom{n-k+r}{n} \binom{r+n}{k-1}$$

	$$ =v^{-r-n-1}\sum_{j=0}^{r-1} v^{-(j+1)(r-1)} (-1)^{j+1-r} \binom{n-j+r-1}{n} \binom{r+n}{j}$$

	$$ =v^{-r-n-1 - (r-1)}\sum_{j=0}^{r-1} v^{-j(r-1)} (-1)^{j-(r-1)} \binom{n-j+r-1}{n} \binom{r+n}{j}$$
	
	$$ = v^{-r-n-1 - (r-1)} \cdot v^{-(r-1)(n+(r-1)+1)} = v^{(-r(n+r+1))},$$ where the penultimate equality follows from the inductive hypothesis.

\end{proof}

\begin{proof}[Proof of Lemma \ref{basicequality}]
	We have by Lemma \ref{moving-Fjs}
	
	$$\f_{i,l}\f_j^{l+1+n}.1 = b_{i,l}F_j^{(l+1+n)}  = \sum_{t=0}^{l} (-1)^{t+l} \binom{l+n-t}{n} F_j^{(l+n+1-t)} b_{i,l} F_j^{(t)}$$
	
	$$ = \sum_{t=0}^{l} (-1)^{t+l} \binom{l+n-t}{n} F_j^{(l+1+n-t)} \sum_{k=0}^{t} F_j^{(k)} z_{k,t}$$
	
	$$ = \sum_{t=0}^l \sum_{k=0}^t (-1)^{t+l} \binom{l+n-t}{n} F_j^{(l+1+n-t+k)} \binom{l+1+n-t+k}{k} z_{k,t}$$

	$$ = \sum_{s=n+1}^{l+1+n} F_j^{(s)} \sum_{k=0}^{s-n-1} (-1)^{l+1+n+k-s+l} \binom{l+n-(l+1+n+k-s)}{n} \binom{s}{k} z_{k,l+1+n+k-s}$$ 

	$$= \sum_{r  =0}^l F_j^{(r+n+1)} z_{0,l-r} \left( \sum_{k=0}^{r} v^{-kr} (-1)^{k-r} \binom{n-k+r}{n} \binom{r+n+1}{k} \right).$$ 
	
	By Lemma \ref{u_{n,t}sum}, the above equals
	$$\sum_{r  =0}^l v^{-r(n+r+1)}F_j^{(r+n+1)} z_{0,l-r}.$$ 
	
	Recall by Lemma \ref{inLinfty} that $z_{0,l-r} \in \L$ for $0 \leq r \leq \ell$. Since $-r(n+r+1) < 0$ for $0 < r \leq l$, we have $\f_{i,l}\f_j^{l+1+n}.1 \equiv F_j^{(n+1)}z_{0,l} = \f_j^{n+1}z_{0,l}$. By Lemma \ref{inLinfty}, $z_{0,l} \equiv \f_{i,l} \f_j^l .1 $, thus we have $\f_{i,l} \f_j^{l+1+n}.1 \equiv \f_j^{n+1}\f_{i,l} \f_j^l .1$. This is precisely the equality we want for $n=0$. If $n > 0$, then similarly we have $\f_{i,l} \f_j^{l+n}.1 \equiv \f_j^n \f_{i,l} \f_j^l.1$. Composing with another $\f_j$ on the left, we get $\f_j \f_{i,l} \f_j^{l+n}.1 \equiv \f_j^{n+1} \f_{i,l} f_j^l.1 \equiv \f_{i,l} \f_j^{l+1+n}.1$ as desired. 
\end{proof}

\subsection{Validity of Crystal Serre Relations}

Having proven that $\f_{i,l} \f_{j}^{l+1+n}.1 = \f_j \f_{i,l} \f_j^{l+n}.1$, we extend this equality to show for any $k \in \L$, $\f_j\f_{i,l}\f_j^l . k \equiv \f_{i,l}\f_j^{l+1} . k$ in $\B$. 

\begin{lemma}\label{opassoc}
	Let $z \in \mathcal{K}_j \subseteq U^{-}$ and let $x \in U^{-}$. Then $\f_j.(xz) = (\f_j. x)z$ and $\f_{i,l}.(xz) = (\f_{i,l} . x)z$ for all $l$.  	
\end{lemma}

\begin{proof}
 By Lemma \ref{decomp}, write $x$ as $\sum_{n} F_j^{(n)} z_n$ where all $z_n$ lie in $\mathcal{K}_j$. Then 
 \begin{equation}
 \label{trans}
 \f_j(xz) = \f_j . \sum_{n} F_j^{(n)} z_nz = \sum_{n} F_j^{(n+1)} z_nz,
 \end{equation} since $\mathcal{K}_j$ is closed under multiplication. But Equation \eqref{trans} 
 
 $$=\left(\sum_{n} F_j^{(n+1)} z_n \right) z = (\f_j . x) z.$$ 
 
The equality $\f_{i,l}(xz) = (\f_{i,l} . x)z$ follows from the associativity of $U^-$ as an algebra, as the left action of $\f_{i,l}$ is simply left multiplication by $b_{i,l}$. 	
\end{proof}



\begin{corollary}\label{rightmultstable}
	$\L$ is stable under right multiplication by elements in $\mathcal{K}_j \cap \L$. 
\end{corollary}

\begin{proof}
	Let $k \in \L$ and $z \in \mathcal{K}_j \cap \L$. Decomposing $k$ into the sum of sequences of Kashiwara operators applied to 1, it suffices to show the statement holds for $\f_\gamma.1$ where $\gamma$ is a sequence of elements of $I_\infty$ and $\f_\gamma$ denotes the associated sequence of Kashiwara operators. By Lemma \ref{opassoc}, we can write $(\f_\gamma.1)z = \f_\gamma(z)$, which is in $\L$ since $\L$ is stable under Kashiwara operators.

\end{proof}

\begin{corollary}\label{partinL}
	When $z \in \L$ is written as $\sum_{n=0}^{k} F_j^{(n)} z_n$ for $z_n \in \mathcal{K}_j$, then in fact $z_n \in \L$ for all $n$. 
\end{corollary}

\begin{proof}
	We use induction on $k$. If $k=0$, then clearly $z = z_0 \in \L$. 
	
	For the inductive step, we consider $\tilde{e}_j^k . \sum_{n=0}^k F_j^{(n)} z_n = z_k \in \L$, as $\L$ is stable under $\e_j$. The result follows by applying the inductive hypothesis to $z - F_j^{(k)} z_k$. 
\end{proof}

\begin{theorem}
	For all $k \in \L$, $\f_{i,l} \f_j^{l+1} . k \equiv \f_j \f_{i,l} \f_j^{l} . k$ in $\B$. 
\end{theorem}

\begin{proof}
	
	By Lemma \ref{basicequality}, we know $\f_{i,l} \f_j^{l+n+1} . 1 \equiv \f_j \f_{i,l} \f_j^{l+n} . 1$ in $\B$ for any $n \geq 0$. Thus we can write $\f_{i,l} \f_j^{l+n+1} . 1 \equiv \f_j \f_{i,l} \f_j^{l+n} . 1 + \alpha$ where $\alpha \in v^{-1} \L$. 
	
	Writing $k$ as $\sum_{n} F_j^{(n)} z_n$ for $z_n \in \mathcal{K}_j$, by linearity it suffices to show that $\f_{i,l} \f_j^{l+1} . F_j^{(n)} z_n \equiv \f_j \f_{i,l} \f_j^{l} . F_j^{(n)} z_n $ in $\B$. This is equivalent to showing that for any $n$, $\f_{i,l} \f_j^{l+n+1} . z_n \equiv \f_j \f_{i,l} \f_j^{l+n} . z_n$. Since $z_n \in \mathcal{K}_j$, by Lemma \ref{opassoc}, $\f_{i,l} \f_j^{l+n+1} . z_n = (\f_{i,l} \f_j^{l+n+1} . 1) z_n  = (\f_j \f_{i,l} \f_j^{l+n}.1 + \alpha) z_n = \f_j \f_{i,l} \f_j^{l+n} . z_n + \alpha z_n$. By Lemma \ref{partinL}, $z_n \in \L \cap \mathcal{K}_j$. Thus by Lemma \ref{rightmultstable}, $\alpha z_n \in v^{-1}\L$ so $\alpha  z \equiv 0$. Thus $\f_{i,l} \f_j^{l+n+1} . z_n \equiv \f_j \f_{i,l} \f_j^{l+n} . z_n$ in $\B$, and the theorem follows.
\end{proof}

\subsection{Kashiwara Operators for Non-adjacent Vertices Commute}

\begin{lemma}\label{js-commute}
	
	Let $\iota, \iota' \in I_\infty$ be such that $(\iota, \iota')=0$. Then for all $u \in U^-$, $\f_{\iota} \f_{\iota'} . u= \f_{\iota'} \f_\iota . u$. 
\end{lemma}

First consider the case where both $\iota, \iota' \in I^{im}$. Since $(\iota, \iota')=0$, $[b_\iota, b_{\iota'}]=0$, and the lemma follows trivially from the fact that the action of $\f_\iota$ and $\f_{\iota'}$ are simply left multiplication by $b_{\iota}$ and $b_{\iota'}$ respectively. If exactly one of $\iota, \iota' \in I^{im}$, without loss of generality assume $\iota \in I^{im}$ and let $\iota' := j \in I^{re}$. Then write $u =\sum F_j^{(n)}z_n$ for $z_n \in  K_j$. Then 
$$\f_\iota \f_j. u =  \sum b_\iota F_j^{(n+1)}z_n =  \sum F_j^{(n+1)} b_\iota z_n,$$ because $[F_j, b_\iota]=0$ as $(\iota,j)=0$. Also, $b_\iota z_n \in K_j$ because $K_j$ is a subalgebra of $U^-$, and thus the above summation equals $$\sum F_j^{(n+1)} b_\iota z_n = \f_j \sum F_j^{(n)} b_\iota z_n = \f_j  \f_\iota. u.$$

It remains to prove Lemma \ref{js-commute} for the case where $\iota, \iota' \in I^{re}$. For the rest of this section, let $\iota = j, \iota' = k$ where $j, k \in I^{re}$. 
\begin{lemma}\label{ej'ek'commute}
	$[e_j', e_k'] = 0$.
\end{lemma}

\begin{proof}[Proof of Lemma \ref{ej'ek'commute}]
	
	It suffices to show that $e_j'(e_k' (b_{\iota_1, \ldots, \iota_n})) = e_k'(e_j' (b_{\iota_1, \ldots, \iota_n}))$ for any sequence $(\iota_1, \ldots, \iota_n)$ where all $\iota_t \in I_\infty$ and $\iota_t= (i_t, c_t)$ for $i_t\in I$, where $b_{\iota_1, \ldots, \iota_n} = b_{\iota_1} \cdots b_{\iota_n}$.
	
	Denote $b_{\iota_2, \ldots, \iota_n}$ by $\alpha$. Then by Property (1) of $e_j'$ and $e_k'$ mentioned in Section \ref{2.2},
	
	$$e_j'(e_k' (b_{\iota_1, \ldots, \iota_n})) = e_j' (e_k' (b_{\iota_1}) \alpha + v^{(-k, -c_1 i_1 )} b_{\iota_1} e_k'(\alpha ) ) = e_j'(e_k' (b_{\iota_1}) \alpha) + v^{(-k, -c_1 i_1)}  e_j'(b_{\iota_1} e_k'(\alpha )  )$$
	
	$$ = e_j'(e'_k(b_{\iota_1})) \alpha + v^{(-j,|e_k'(b_{\iota_1})|)}e_k'(b_{\iota_1}) e_j'(\alpha) + v^{(-k, -c_1 i_1)} \left( e_j'(b_{\iota_1})e_k'(\alpha) + v^{(-j, -c_1i_1)}b_{\iota_1}e_j'(e_k'(\alpha))\right) $$
	
	Since $e_k'$ is a homogeneous operator of degree $k$, we have that $|e_k'(b_{\iota_1})| = -c_1i_1 + k$. Thus $(-j, -c_1i_1+k) = (-j, -c_1i_1)$ since $(j,k)=0$.
	

	The middle two terms of the expression are symmetric in $j$ and $k$ because $(j,k)=0$. $e_j'(e_k'(b_{\iota_1})) = e_k'(e_j'(b_{\iota_1}))$ for degree reasons. The last term is symmetric in $j$ and $k$ because we can assume by induction on $n$, the length of the product, that $e_j'$ and $e_k'$ commute on $\alpha$. Thus the entire expression is symmetric in $j$ and $k$, and thus $e_j'$ and $e_k'$ commute.
\end{proof}

It follows that $\mathcal{K}_j$ is stable under the $e_k'$ operator and $\mathcal{K}_{k}$ is stable under the $e_j'$ operator.

\begin{lemma}\label{fin}
	
Given $z \in \mathcal{K}_k$, writing $z = \sum_{n=0}^d F_{j}^{(n)} z_n$ where $z_n \in \mathcal{K}_{j}$, then $z_n \in \mathcal{K}_k$. 
\end{lemma}

\begin{proof}
	By induction on $d$, it suffices to check $z_d \in \mathcal{K}_k$. because $F_j$, and then also $F_j^{(d)}z_d$, is in $K_k$ It follows from Lemma \ref{ej'ek'commute} that $e_j'^d (z) \in \mathcal{K}_k$. But it is also clear from the fact that $e_j'$ has the skew derivation property (1) given in Section 2 that $e_j'^d(z) = cz_d$ with $c \in \mathbb{C}(v)$ a non-zero constant, and the claim follows.
\end{proof}

\begin{proof}[Proof of Lemma \ref{js-commute}]
	
	Let $u \in U^{-}$. Then write $u = \sum_{n} F_j^{(n)} z_n$ with $z_n \in \mathcal{K}_j$ and write each $z_n = \sum_{m} F_k^{(m)} z_{n,m}$ with $z_{n,m} \in \mathcal{K}_k$. By Lemma \ref{fin}, $z_{n,m} \in \mathcal{K}_j$ as well. Because  $z_{n,m} \in \mathcal{K}_k \cap \mathcal{K}_j$, $\f_j \f_k . u = \sum_{n,m} F_j^{(n+1)} F_k^{(m+1)} z_{n,m}$ since $F_j \in \mathcal{K}_k$ and $F_k \in \mathcal{K}_j$. But because $[F_j, F_k]=0$, we get that
		$$\sum_{n,m} F_j^{(n+1)} F_k^{(m+1)} z_{n,m} = \sum_{n,m}  F_k^{(m+1)}F_j^{(n+1)} z_{n,m} = \f_k \f_j . u,$$ as needed.
\end{proof}

\section{Sufficiency of Commutation and Crystal Serre Relations}\label{combo_charform}

In this section we finish the proof of Theorem \ref{maintheorem} by showing that the commutation and crystal Serre relations proved in Section 3 imply all equivalences $\f_{\iota_1}...\f_{\iota_n}.1 \equiv \f_{\iota'_1}...\f_{\iota'_m}.1$ in $\B$ for the quiver $Q(\omega, r)$ with $\omega > 1$ and $r \geq 0$. To achieve this we give a combinatorial analysis of Bozec's formula for the graded dimension of $U^-$.

\subsection{Character Formula}

Similar to the classical case, in \cite{addendum} Bozec gives an explicit character formula for the graded dimension of the algebra $U^-$ associated to any finite quiver $Q$.  The formula and its proof are analogous to the case of Kac-Moody algebras, for example in $\S11.13$ of \cite{Kac}.  We state here this character formula in the special case of the quiver $Q(\omega, r)$ for $\omega > 1, r \geq 0$, which we denote as $\Ch U^-$. 

Fix $r \geq 0$.  Let $x_1, ..., x_r, y$ be commuting indeterminates, let $R:= \{ x_1, \ldots, x_r\}$, let $\mathcal{P}(R)$ denote the powerset of $R$, and for a subset $S \subset R$ let $\pi(S)$ denote the product of all elements in $S$. Then define

$$(\Ch U^-)^{-1} := \sum_{p \in \mathcal{P}(R)} (-1)^{|p|} \pi(p) \left( 1-\frac {\pi(p)y}{1-\pi(p)y} \right).$$

The coefficient of $y^{n} x_1^{m_1} \cdots x_r^{m_r}$ in the power series $\Ch U^-$ gives the number of elements in $\B$ of degree $-ni - m_1j_1 - \ldots - m_rj_r$ for the quiver $Q(\omega, r)$ for any $\omega > 1$.

\subsection{A Combinatorial Description}

	\begin{definition}
		
		Let $\vec{\iota}$ be a finite sequence of elements of $I_\infty$. $\vec{\iota}$ is called \emph{steep} if it is of the following form: 
		
		$$\vec{\iota} = (j_1^{p_{0,1}}, \ldots, j_r^{p_{0,r}}, (i, c_1), j_1^{p_{1,1}},, \ldots, j_r^{p_{1,r}}, (i, c_2), \ldots, (i,c_n), j_1^{p_{n,1}}, \ldots j_r^{p_{n,r}}), $$ where the notation $j_k^{p_{i, k}}$ indicates $p_{i, k}$ successive occurrences of $j_k$, and where $c_i \geq p_{i,k}$ for $1 \leq i \leq n$ and $1 \leq k \leq r$. \end{definition}
	
	\begin{definition}
		
		Given a sequence $\vec{\iota} = (\iota_1, \ldots, \iota_n)$ of elements of $I_\infty$, define $\f_{\vec{\iota}} := \f_{\iota_1} \circ \cdots \circ \f_{\iota_n}$. 		
	\end{definition}
	
	\begin{lemma} \label{steep-transform}
		For any finite sequence $\gamma$ of elements of $I_\infty$, there is a steep sequence $\gamma'$ such that $\f_\gamma \equiv \f_{\gamma'}$ on $\B$. 
	\end{lemma}
	
	\begin{proof}
		
	The lemma follows from the commutation relations and the crystal Serre relations proved in Section 3.
	\end{proof}
	
	\begin{definition} The \emph{degree} of a sequence $\vec{\iota} = (\iota_1, ..., \iota_n)$ of elements of $I_\infty$ is $|\vec{\iota}| := -\sum_{k = 1}^n |\iota_k|$, the degree of the associated composition of Kashiwara operators $\f_{\vec{\iota}}$.\end{definition}
	
	\begin{lemma}\label{number-steep}
		
		The number of steep sequences of a given degree $\mu$ is equal to $\#\B[\mu]$. 
	\end{lemma}

	\begin{proof}

		We derive a recursion for the coefficients of $\Ch U^-$ and show that the number of steep sequences of a particular degree also satisfy the same recursion and initial values.  
		
		Notice that  $r \Ch U^- = \Ch U^- -1$ where
		
		$$ r = 1- (\Ch U^-)^{-1} = \frac {y}{1-y}  -  \sum_{p \in \mathcal{P}(R) \setminus \emptyset} (-1)^{|p|} \pi(p) \left( 1-\frac {\pi(p)y}{1-\pi(p)y} \right).$$

		The equality $r \Ch U^- = \Ch U^- - 1$ can be interpreted as a recursion on the coefficients of the power series because $r$ has zero constant term. We let $c(n, \vec{m}):= c(n,m_1, \ldots, m_r)$ denote the coefficient of $y^{n} x_1^{m_1} \cdots x_r^{m_r}$ in $\Ch U^-$. We now make this recursion explicit.
		Note
		
		$$r = \sum_{k \geq 1} y^k - \sum_{p \in P(R) \setminus \emptyset} (-1)^{|p|} \pi(p)  + \sum_{k \geq 1} \sum_{p \in P(R) \setminus \emptyset} (-1)^{|p|} \pi(p)^{k+1} y^k.$$
		
		For $p \in \mathcal{P}(R)$, let $\vec{p}$ denote the tuple $(\delta_{1}, \ldots, \delta_r)$ where $\delta_i = 1$ if $x_i \in p$ and $0$ otherwise. Then $r\Ch U^- = \Ch U^- -1 $ implies for $(n, \vec{m}) \neq (0,0)$, 
		
		\begin{equation}\label{recursion}
		c(n,\vec{m}) = \sum_{k \geq 1} c(n-k,\vec{m}) - \sum_{p \in P(R) \setminus \emptyset}  (-1)^{|p|} c(n, \vec{m} - \vec{p}) + \sum_{k \geq 1} \sum_{p \in P(R) \setminus \emptyset} (-1)^{|p|}c(n-k,\vec{m}-(k+1)\vec{p}).
		\end{equation}
		
		Note $c(0,\vec{0}) = 1$, and for negative $n$ or negative $m_i$, $c(n,\vec{m}) = 0$. Thus the number of steep sequences of a particular degree satisfies the same intial values.
		
		
		We explain why the recursion given in Equation \eqref{recursion} also gives a recursion for the number of steep sequences of a given degree. There are 2 cases to consider. In the first case, the sequence ends with $\f_{j_i}$ for some $r \geq i \geq 1$. Since the $\f_{j_i}$ commute by Lemma \ref{js-commute}, the number of sequences ending with a $\f_{j_i}$ can be counted with the principle of inclusion-exclusion, noting that if a steep sequence ends with $\f_{j_i}$ then it remains steep. This corresponds to 
		
		$$- \sum_{p \in P(R) \setminus \emptyset}  (-1)^{|p|} c(n, \vec{m} - \vec{p}).$$

		If the sequence ends with a $\f_{i,l}$, which is a disjoint condition from ending with an $\f_{j_i}$, this corresponds to
		
		$$\sum_{k \geq 1} c(n-k,\vec{m}).$$
		
		However, we must also subtract the steep sequences to which concatenating a $\f_{i,l}$ makes the sequence not steep. Such sequences end with $\f_{j_1}^{p_1} \cdots \f_{j_r}^{p_r}$ where for at least 1 index $k$, $p_k > l$. Such sequences are counted by 
		
		$$\sum_{k \geq 1} \sum_{p \in P(R) \setminus \emptyset} (-1)^{|p|}c(n-k,\vec{m}-(k+1)\vec{p})$$
		
		The recursion in Equation \eqref{recursion} follows. Since the number of steep sequences satisfy the same base case and recursion as $\Ch U^-$, we have shown the lemma. 
	\end{proof}
		
		
		
		
	
	We can now finish the proof of our main theorem (Theorem \ref{maintheorem}):
	
	\begin{proof}[Proof of Theorem \ref{maintheorem}] We saw in Section 3 that the commutation and crystal Serre relations of Theorem \ref{maintheorem} hold.  It follows, as in Lemma \ref{steep-transform}, that every element $\f_{\vec{\iota'}}.1$ of $\B$ is equal by a series of applications of the commutation and crystal Serre relations to an element $\f_{\vec{\iota}}.1$ for a steep sequence $\vec{\iota}$.  It follows by Lemma \ref{number-steep} that, for any degree $\mu$, the elements $\f_{\vec{\iota}}.1$ for $\vec{\iota}$ a steep sequence of degree $\mu$ are pairwise inequivalent and comprise every element of $\B[\mu]$.  The theorem follows.\end{proof}
		 
		
		
		
		


\begin{thebibliography}{9}
	
	\bibitem{HaLeRo}
	
	T. Hausel, E. Letellier, F. Rodriguez-Villegas.
	\emph{Arithmetic harmonic analysis on character and quiver varieties II.}
	Adv. Math., 234:85Ð128, 2013.
	
	\bibitem{Ka}
	
	M. Kashiwara.
	\emph{On crystal bases of the q-analogue of universal enveloping algebras.}
	Duke Math. J. 63 (1991), no. 2, 465--516.
	
	\bibitem{Lu}
	
	G. Lusztig.
	\emph{Quivers, perverse sheaves, and quantized enveloping algebras.}
	J. Amer. Math. Soc., 4(2):365-421, 1991.
	
	
	
	
		
	\bibitem{Bo1}
		
		T. Bozec.
		\emph{Quivers with loops and perverse sheaves.}
		Mathematische Annalen, 1--25, 2013.
		
		
	\bibitem{addendum}
	
		T. Bozec. 
		\emph{Addendum to `quivers with loops and perverse sheaves'.}
		arXiv: 1509.09302v1, 2015.
			
	\bibitem{Bo2}
	
	T. Bozec.
	\emph{Quivers with loops and generalized crystals.}
	arXiv: 1403.0846v3, 2014.
	
	\bibitem{Kac}
	
	V. Kac.
	\emph{Infinite-Dimensional Lie Algebras}
	Cambridge University Press, 1994.
	
	
	
	
	
	
\end{thebibliography}
\end{document}